\documentclass{amsart}
\usepackage{graphicx}
\usepackage{float}
\usepackage{amsmath}
\usepackage{amssymb}
\usepackage{enumerate}
\usepackage{verbatim}
\usepackage{caption}
\usepackage{physics}
\usepackage{url}
\usepackage{textcomp}
\usepackage{color}
\usepackage{hyperref}
\usepackage{multicol}
\usepackage{wrapfig}
\usepackage{gensymb}
\usepackage{subcaption}
\usepackage{bm}
\usepackage{CJKutf8}
\usepackage{multirow}
\usepackage{tabularx}
\usepackage{transparent}
\usepackage[utf8]{inputenc}
\usepackage{amssymb}
\usepackage{amsthm}
\usepackage{mathtools}
\numberwithin{equation}{section}
\allowdisplaybreaks
\usepackage{enumitem}
\newcounter{descriptcount}
\newlist{enumdescript}{description}{1}
\setlist[enumdescript,1]{
  before={\setcounter{descriptcount}{0}
          \renewcommand*\thedescriptcount{\arabic{descriptcount}}},
        font={\bfseries\stepcounter{descriptcount}Case \thedescriptcount~}
}
\allowdisplaybreaks

\newtheorem{lemma}{Lemma}[section]
\newtheorem{theorem}[lemma]{Theorem}
\theoremstyle{definition}

\theoremstyle{plain}
\newtheorem{prop}[lemma]{Proposition}

\newtheorem{coro}[lemma]{Corollary}
\usepackage[abbrev,nobysame]{amsrefs}
\begin{document}

\title{Attractors of linear maps with bounded noise
}

\author{Jeroen S.W. Lamb}
\address{Department of Mathematics, Imperial College London, 180 Queen's Gate, London SW7 2AZ, United Kingdom}
\address{International Research Center for Neurointelligence (IRCN), The University of Tokyo, 7-3-1 Hongo Bunkyo-ku, Tokyo, 
113-0033 Japan}
\address{Centre for Applied Mathematics and Bioinformatics, Department of Mathematics and Natural Sciences, Gulf University for Science and Technology, Halwally 32093, Kuwait}

\email{jsw.lamb@imperial.ac.uk}
\thanks{}

\author{Martin Rasmussen}
\address{Department of Mathematics, Imperial College London, 180 Queen's Gate, London SW7 2AZ, United Kingdom}
\email{m.rasmussen@imperial.ac.uk}
\thanks{}

\author{Wei Hao Tey}
\address{International Research Center for Neurointelligence (IRCN), The University of Tokyo, 7-3-1 Hongo Bunkyo-ku, Tokyo, 
113-0033 Japan}

\email{weihaotey@g.ecc.u-tokyo.ac.jp.}
\thanks{The first two authors were supported by the EPSRC grant EP/W009455/1. The third author was supported by the Project of Intelligent Mobility Society Design, Social Cooperation Program, UTokyo, JST Moonshot R \& D Grant Number JPMJMS2021, EPSRC grant EP/S515085/1}



\begin{abstract}
   We consider invertible linear maps with additive spherical bounded noise. We show that minimal attractors of such random dynamical systems are unique, strictly convex and have a continuously differentiable boundary. Moreover, we present an auxiliary finite-dimensional deterministic \emph{boundary map} for which the unit normal bundle of this boundary is globally attracting. 
\end{abstract}

\maketitle

\section{Introduction and main results}
The development of  dynamical systems theory is one of the scientific revolutions of the 20th century, and many important insights from this field are now embedded within computational methodologies in many branches of science and engineering as well as at the heart of deep and abstract mathematics. During the last few decades, the importance of noise and uncertainty has become evident in real-world applications of dynamical systems, but the corresponding random dynamical systems theory is still only in its early stages of development~\cite{arnold98}.

In random dynamical systems with bounded noise, trajectories are typically attracted by bounded invariant sets, from which escape is not possible. These attractors are important for the understanding of the long-term dynamical behaviour. In particular, they support ergodic stationary measures which describe statistical features of the random dynamics \cite{zmarrou2007bifurcations}.
It has been observed \cite{Lamb2023numerical} that nonlinear random dynamical systems with bounded noise
may have multiple coexisting attractors with complicated shapes whose boundaries may have points of non-differentiability, yielding the general study of such attractors challenging \cites{Kourliouros2023,lamb2020boundaries,lamb2021boundaries}.

In this paper, we study geometric properties of minimal invariant sets for the elementary class of 
random dynamical systems consisting of invertible linear maps with spherical bounded noise. 
We choose this setting since linear random dynamical systems provide insight into attractors when the underlying dynamics is approximately linear, like when perturbing a nonlinear system near a periodic sink with small bounded noise.

We study iterations in $\mathbb{R}^m$ of the form
\begin{equation}\label{eq:rds}
    x_{i+1} = L(x_i) + \xi_i\,,
\end{equation}
where $i\in\mathbb{N}_0$, $x_i\in\mathbb{R}^m$ and $L\in GL(m,\mathbb{R})$ is an invertible linear map. $\{\xi_i\}_{i\in\mathbb N_0}$ is a sequence of i.i.d.~random variables that are supported on a ball of radius $\varepsilon$ centered at~$0$.

We say that a compact subset $A\subset\mathbb{R}^m$ is an attractor of the random dynamical system \eqref{eq:rds}, if it is attracting and minimally forward invariant.\footnote{
Attracting random subsets for specific noise realisations are referred to as random attractors  cf.~\cite{arnold98}.
The attractors defined here are deterministic and the union of random attractors over all noise realisations.} $A$ is minimally forward invariant if $x_i\in A$ implies that $x_{i+1}\in A$ and there does not exist a proper subset of $A$ that is also forward invariant.
$A$ is attracting if there exists a neighbourhood $N$ of $A$ such that all forward orbits starting in $N$ tend to $A$, in the sense that  $\lim_{i\to\infty}\inf_{a\in A} d(x_{i},a)=0$ for all $x_0\in N$, with $d$ denoting the Euclidean metric in $\mathbb{R}^m$. 
The long term behaviour of \eqref{eq:rds} ($i\to\infty$) is characterized by its dynamics on attractors.

In this paper, we establish
that \eqref{eq:rds} has an attractor if and only if $L$ is eventually contracting, i.e.~if all its eigenvalues lie inside the unit circle, and that in this case the attractor is unique and globally attracting (in the sense that forward obits of all initial conditions tend to the attractor). 
Moreover, concerning the shape of minimal attractors of \eqref{eq:rds}, we show that these are strictly convex sets with a continuously differentiable boundary. 

It remains an open problem whether there exists a more concise geometric, analytical or algebraic characterisation of these attractors. For instance, we note that the attractor is a solid sphere of constant radius if and only if $L$ is a scalar multiple of the identity. Interestingly, this is also the only instance in which the attractor is a solid ellipsoid \cite{homburg2010bifurcations}.  

This paper is organized as follows. We establish existence and uniqueness in Section~\ref{sec:eau}, by taking an infinite-dimensional set-valued approach, following \cite{lamb2015topological}. 

In Section~\ref{sec:global} we show that a finite-dimensional boundary map, first introduced in \cites{Kourliouros2023,tey2022minimal}, has a globally attracting object, representing the unit normal bundle of the boundary of the unique attractor of \eqref{eq:rds}. The latter boundary map helps us to establish the strict convexity of the attractor but serves also as an efficient computational tool to approximate the boundary of attractors.

\section{Existence and uniqueness}\label{sec:eau}
To understand the structure of attractors of \eqref{eq:rds}, it suffices to consider the compound evolution of trajectories of \eqref{eq:rds}, represented by the set-valued map $L_\varepsilon(x)=\overline{B_\varepsilon(L(x))}$, where $B_\varepsilon(x):=\{y\in\mathbb{R}^m~|~d(x,y)<\varepsilon\}$ denotes the open ball of radius $\varepsilon$. This map naturally extends to the space $\mathcal{K}(\mathbb{R}^m)$ of non-empty compact subsets of $\mathbb{R}^m$ as
\begin{equation*}
    L_\varepsilon(C) := \overline{B_{\varepsilon}(L(C))} = \{x+y \in \mathbb{R}^m \mid x \in L(C), \|y\|\le \varepsilon\},
\end{equation*}
where $L(C):=\{L(x)~|~ x\in C\}$.

There exists a useful correspondence between attractors of \eqref{eq:rds} and attracting fixed points of $L_\varepsilon$. To give context to the latter notion, we consider $\mathcal{K}(\mathbb{R}^m)$ equipped with the Hausdorff metric
\[
d_H(A,B):=\max\{\sup_{x\in A} \inf_{y\in B} d(x,y),
\sup_{y\in B} \inf_{x\in A} d(x,y)\}.
\]
We say that $A$ is an attracting fixed point of $L_\varepsilon$ if $L_\varepsilon(A)=A$ and $\lim_{i\to\infty}L_\varepsilon^i(C)= A$ for all $C\in \mathcal{K}(\mathbb{R}^m)$ in some open neighbourhood of $A$, and it is called globally attracting if the convergence holds for all $C\in\mathcal{K}(\mathbb{R}^m)$.

\begin{lemma}\label{lem:rdsvssetv} If $A\in \mathcal{K}(\mathbb{R}^m)$ is an attractor of \eqref{eq:rds} then $A$ is a fixed point of $L_\varepsilon$. If $A$ is a globally attracting fixed point of $L_\varepsilon$, then $A$ is the unique global  attractor of \eqref{eq:rds}.
 
\end{lemma}
\begin{proof}
Let $A$ be an attractor of \eqref{eq:rds}. We show that $L_\varepsilon(A)=A$. Let $A^c:=\mathbb{R}^m\setminus A$. Then $L_\varepsilon(A)\cap A^c = \emptyset$, otherwise there exists $x_i\in A$ such that $x_{i+1}\not\in A$. Moreover, $A\setminus L_\varepsilon(A)=\emptyset$, otherwise $A$ would not be a minimal trapping region, since $L_\varepsilon(A)\subseteq A$ implies that
$L_\varepsilon^i(A)\subseteq L_\varepsilon(A)$ for all $i\in\mathbb{N}$.       
The second claim holds, as $\lim_{i\to\infty}L_\varepsilon^i(C)=A$ for all $C\in\mathcal{K}(\mathbb{R}^m)$ implies that
$\lim_{i\to\infty}\inf_{a\in A} d(x_{i},a)=0$ for all $x_0\in A^c$.
\end{proof}

The existence and uniqueness of an attractor for \eqref{eq:rds}and its global attractivity, follow from Lemma~\ref{lem:rdsvssetv} and the following proposition.
%point for $L_\varepsilon$ 
\begin{prop}[Existence and uniqueness]\label{prop:unique}
  $L_\varepsilon$ has a fixed point $A\in \mathcal{K}(\mathbb{R}^m)$ if and only if all eigenvalues of $L$ have magnitude less than 1. The fixed point $A$ is unique and globally attracting for $L_\varepsilon$.
\end{prop}

\begin{proof}

The condition on the eigenvalues of $L$ implies that there exists a metric $d_w$ on $\mathbb{R}^m$, equivalent to the Euclidean metric, such that $L$ is a contraction (i.e.~there exists $0<c<1$ such that $d_w(L(x),L(y))\leq c d(x,y)$ for all $x,y\in\mathbb{R}^m$) \cite{belitskii2013matrix}*{Corollary~2.1.1}. If $L$ is a contraction on $\mathbb{R}^m$ with metric $d_w$  then so is its set-valued analogue $L_\varepsilon$ on $\mathcal{K}(\mathbb{R}^m)$, equipped with the associated Hausdorff metric \cite{lamb2015topological}*{Lemma~3.3}. 
Finally, as $\mathcal{K}(\mathbb{R}^m)$ is complete%\cite{completeness}
, if $L_\varepsilon$ is a contraction, then by Banach's Fixed Point Theorem $L_\varepsilon$ has a unique fixed point $A$ and $\lim_{i\to\infty}L_\varepsilon^i(C)=A$, for any $C\in\mathcal{K}(\mathbb{R}^m)$.

%($\Rightarrow$) 
It remains to be shown that the condition on the eigenvalues of $L$ is necessary. Suppose $L$ has an eigenvalue $\lambda$ with $|\lambda|>1$, then for every $x$ in the corresponding (real) eigenspace $E_\lambda$, $\|L^i(x)\|= |\lambda|^i\|x\|$. Hence,  $\{\|L^i(x)\|~|~i\in\mathbb{N}\}$ is unbounded. 
We show that this is incompatible with the existence of a bounded attractor, exploiting the linearity of $L$. Namely, every attractor $A$ contains a point $y\in A$ such that $\overline{B_\varepsilon(y)}\subset A$. This ball contains a point $z$, whose direct sum into elements of (generalized) eigenspaces of $L$, contains a nonzero component in $E_\lambda$. This implies, by $L$-invariance of $E_\lambda$, that $\{\|L^i(z)\|~|~i\in\mathbb{N}\}$ is unbounded, which contradicts the boundedness of the attractor $A$. 

If $L$ has an eigenvalue $\lambda$ with $|\lambda|=1$, by the preceding argument, there exists a point $z$ in the attractor with nonzero component, say $x$, in the (real) eigenspace $E_\lambda$. In this case, $\{\|L^i(x)\|~|~i\in\mathbb{N}\}$ is not unbounded, as $\|L(x)\|=\|x\|$ if $x\in E_\lambda$. However, we observe that  $(1+\frac{\varepsilon}{\|L(x)\|})L(x)\in L_\varepsilon(x)$. This implies, by induction,  that $L_\varepsilon^i(x)$ contains a point of distance $i\varepsilon$ from the origin, implying that $\{\|L_\varepsilon ^i(x)\|~|~i\in\mathbb{N}\}$ is unbounded, contradicting the assumption that $z$ lies inside a bounded attractor of $L_\varepsilon$. 
\end{proof}

In the remainder of this paper, we assume that all eigenvalues of $L$ have norm less than 1, and let $A$ denote the unique fixed point of $L_{\varepsilon}$.

\section{Convexity and differentiability}\label{sec:convex}

In this section, we show that the minimal attractor $A$ of \eqref{eq:rds} (which is equal to the unique fixed point of $L_\varepsilon$) is convex, its boundary is always continuously differentiable. 

\begin{prop}[Convexity]\label{prop:conv}
$A$ is convex.
\end{prop}
\begin{proof}
  $A$ contains the origin $0$, since this point is a global attractor of $L$. By induction, it follows that $L^i_{\varepsilon}(0) \subset L^{i+1}_{\varepsilon}(0) \subset A$ for all $i \in \mathbb N_0$. By Proposition~\ref{prop:unique}, the attractor $A$ is the limit of this nested sequence: $A=\lim_{i\to\infty}L^i_{\varepsilon}(0).$ 
  All elements of this sequence 
  are convex. Namely, $L_\varepsilon(0)=\overline{B_\varepsilon(0)}$ is convex, and for any convex $C\in\mathcal{K}(\mathbb{R}^m)$ containing the origin $0$, both  $L(C)$ and $\overline{B_{\varepsilon}(C)}$ are convex. The proof concludes by noting that the limit of a set of nested convex sets is convex. 
\end{proof}

\begin{prop}[Differentiability]
\label{prop:diff}
   The boundary of $A$ is continuously differentiable.
\end{prop}
\begin{proof}
We make use of so-called supporting hyperplanes. Recall that a hyperplane is an $(n-1)$-dimensional affine subspace of $\mathbb{R}^m$, dividing $\mathbb{R}^m$ into two closed halfspaces, $H^-$ and $H^+$, bounded by $H$. A hyperplane $H(x)$ is a supporting hyperplane of a set $C\in\mathcal{K}(\mathbb{R}^m)$ at a boundary point $x \in \partial C$ if either $C \subset H^-$ or $C \subset H^+$. 
Importantly, if $C$ is convex, then there exists a supporting hyperplane at every boundary point
\cite{schneider2014convex}*{Theorem~1.3.2}. Moreover, $\partial C$ is a $C^1$ submanifold of $\mathbb{R}^m$ if and only if the supporting hyperplanes are unique \cite{schneider2014convex}*{Theorem~2.2.4}.

It thus suffices to demonstrate the uniqueness of the supporting hyperplane at every point $x\in \partial A$. 
    Since $\overline{B_{\varepsilon}(L(A))} = A$ and $L$ is invertible, the boundary $\partial A$ of $A$ satisfies $ \partial A = \partial B_{\varepsilon}(L(A))$ $\subset \partial B_{\varepsilon}(L(\partial A))$. Hence, for all $x \in \partial A$, there exists $y \in L(\partial A)$ such that $x \in \partial B_{\varepsilon}(y)$. 
 As $B_\varepsilon(y)\subset A$, the supporting hyperplane at $x$ must also be a supporting hyperplane for $B_\varepsilon(y)$. As the latter is bounded by a smooth manifold, this hyperplane is unique, 
 equal to the plane tangent to $B_\varepsilon(y)$ at $x$. 
\end{proof}

It turns out that $A$ is even strictly convex. This refinement of Proposition~\ref{prop:conv} is presented in Proposition~\ref{prop:strict_convex}. Its proof uses the so-called boundary map, which is introduced and discussed in the next section.

\section{The boundary map and its dynamics}\label{sec:global}
In this section we introduce the \emph{boundary map}, as an alternative tool to study the boundary of the attractor $A$. This map was first introduced in \cite{Kourliouros2023}. We provide a self-contained introduction, tailored to the specific context of this paper.

The observation that leads to the consideration of a boundary map, is that if the boundary of $A$ is smooth, the set-valued map $L_\varepsilon$ induces an invertible map  of $\partial A$ to itself.
\begin{lemma}\label{lem:bminv}
The self-mapping of $\partial A$ induced by $L_\varepsilon$ is invertible.  
\end{lemma}
\begin{proof}
Recall the proof of Proposition~\ref{prop:diff}, where a relationship was established between a point $x\in \partial A$ and a point $y\in L(A)$ such that $x\in\partial B_\varepsilon(y)$. As $L$ is invertible, the restriction of $L$ from $\partial A$ to $\partial L(A)$ is also invertible. If the relationship between $y\in L(A)$ and $x\in \partial A$ is one-to-one, then the invertibility of the restriction of $L_\varepsilon$ to $\partial A$ is established.

To establish the one-to-one relationship, suppose for contradiction that there are two diferent points $y,y'\in L(A)$ whose $\varepsilon$-balls contribute to the same point $x\in\partial A$. 
Since there is only one supporting hyperplane at $x$ for $A$, and $y\neq y'$, $y$ and $y'$ can only lie on opposite sides of $x$, i.e.~$y'=2x-y$. But the combination of two $\varepsilon$-balls touching each other from opposite sides, implies that $\partial A$ would be non-convex.
\end{proof}

\begin{figure}[hb]
    \centering
    \includegraphics[width=.5\textwidth]{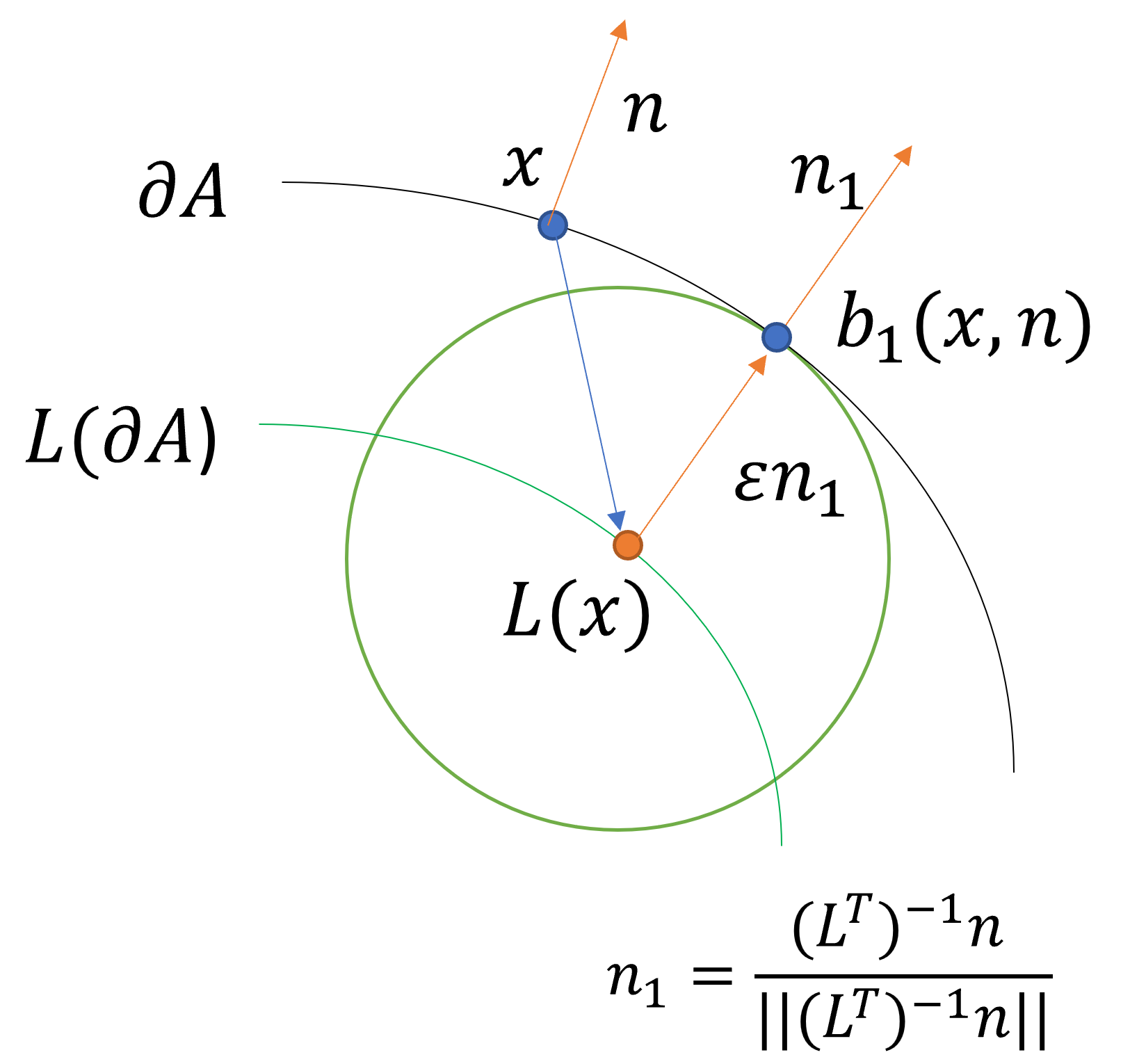}
    \caption{Illustration of the boundary mapping on (normal bundle of) the boundary of the attractor.}
    \label{fig:boundary_map}
\end{figure}

The induced action of $L_\varepsilon$ on the boundary of the attractor does not provide any information about the shape of the boundary. Information about the tangent space (supporting hyperplane) to the boundary, contains additional geometric information about the boundary. Tangent planes are naturally represented by their direction, a unit normal vector. This leads us to define a boundary map from the unit normal bundle $\mathbb{R}^m\times\mathbb{S}^{m-1}$ to itself, where $\mathbb{R}^m$ represents the ambient space of the attractor boundary $\partial A$ and $\mathbb{S}^{m-1}:=\{x\in\mathbb{R}^m~|~\|x\|=1\}$ the sphere of unit vectors in $\mathbb{R}^m$, representing normal directions.
The boundary of the attractor $\partial A$ has a natural representation in $\mathbb{R}^m\times\mathbb{S}^{m-1}$ as the section
\[
N_1^+\partial A:=\{(x,n(x))~|x\in\partial A\},
\]
where $n(x)\in \mathbb{S}^{m-1}$ is the (unique) unit normal to 
$\partial A$ at $x\in\partial A$.

The boundary map $b:\mathbb{R}^m\times\mathbb{S}^{m-1}\to \mathbb{R}^m\times\mathbb{S}^{m-1}$ is defined so as to leave the (outward) unit normal bundle $N^+_1\partial A$ of $\partial A$ invariant. Like $L_\varepsilon$, it is the composition of two maps
\begin{equation}
b(x,n):=t_\varepsilon\circ {L}(x,n), 
\end{equation}
Here, with slight abuse of notation, $L:\mathbb{R}^m\times\mathbb{S}^{m-1}\to \mathbb{R}^m\times\mathbb{S}^{m-1}$ denotes the natural extension of $L\in GL(m,\mathbb{R})$ to the unit normal bundle, in the sense that
for any compact convex set $C$ with smooth boundary $\partial C$, ${L}(N_1^+\partial C)=N_1^+\partial L(C)$. $t_\varepsilon$ denotes $\varepsilon$-translation induced by $\varepsilon$-balls, mapping $N_1^+\partial L(C)$ to $N_1^+\partial L_\varepsilon(C)$. 

We proceed to discuss the action of $L$ on the unit normal bundle. We may write  
\begin{equation}\label{eq:b1}
{L}(x,n):=(L(x), L_\perp(n))
\end{equation}
reflecting in the first component the basic action of $L$ on $\mathbb{R}^m$. The induced action $L_\perp$ on the space $\mathbb{S}^{m-1}$ of unit  normal vectors is given by
\begin{equation}\label{eq:b2}
 L_\perp(n):=P\circ(L^T)^{-1}n,   
\end{equation}
where $P:\mathbb{R}^m\setminus \{0\}\to \mathbb{S}^{m-1}$ denotes the canonical projection to the unit sphere, $P(x)=\frac{x}{||x||}$, and $L^T$ the transpose of $L$. 
To verify that the action on the unit normal induced by $L$ is indeed $L_\perp$, we note that the normal direction is defined to be orthogonal to a tangent space $T$, i.e.~$\langle v,n\rangle=0$ for all $v\in T$.
By linearity, $L$ acts on the tangent space $T$ in the same way as it acts on $\mathbb{R}^m$. Thus the  image of the normal, $L_\perp(n)$, satisfies $\langle L(v),L_\perp(n)\rangle=0$ for all $v\in T$. This implies that 
$L^T L_\perp(n)=n$, from which the expression in \eqref{eq:b2} follows.

The $\varepsilon$-translation $t_\varepsilon$ from $\partial L(A)$ to $\partial A$ is precisely in the direction of the (outward) normal $L_\perp(x)$ of $\partial L(A)$ at $L(x)$:
\begin{equation}
t_\varepsilon(x,n)=(x+\varepsilon n,n).    
\end{equation}

In summary, the boundary map is thus derived to have the form 
\begin{equation}\label{eq:bmap}
b(x,n)=(L(x)+\varepsilon L_\perp(n), L_\perp(n)).
\end{equation}
Importantly, by construction, the unit normal bundle of the attractor boundary $N_1^+\partial A$ is $b$-invariant. See also Figure~\ref{fig:boundary_map}.

We use the boundary map to strengthen the result in Proposition~\ref{prop:conv} and show that $A$ is strictly convex.
\begin{prop}\label{prop:strict_convex}
$A$ is strictly convex.
\end{prop}
\begin{proof}
    Assume for contradiction that $A$ is not strictly convex. Then there exist two distinct boundary points $x,y \in \partial  A$ with the same outward unit normal $n$. Consider the inverse boundary mapping $b^{-1}(x,n)=(L^{-1}(x-\varepsilon n),L_\perp(n))$. Then $d(b^{-1}(x,n),b^{-1}(y,n))=d(L^{-1}(x),L^{-1}(y))$. As $L$ is eventually contracting, $L^{-1}$ is eventually expanding, so that 
    $\lim_{k\to\infty}d(L^{-k}(x),L^{-k}(y))=\infty$, which contradicts the boundedness of $A$.
\end{proof}

It turns out that $N_1^+\partial A$ exhibits global attractivity under the boundary map $b$ and admits an explicit series expansion.
\begin{theorem}[Global attraction of the normal bundle]\label{thm:linear}
 \begin{equation*}\label{eq:att}
  N_1^+\partial A=\{(x(n),n)~|~n\in\mathbb{S}^{m-1}\},\mbox{~with~}
x(n):=\varepsilon\sum_{k=0}^\infty L^k\circ L_\perp^{-k}(n),
  \end{equation*}
  is globally attracting under the boundary map $b$. 

\end{theorem}

\begin{proof}

    Because of the strict convexity of $A$ (Proposition~\ref{prop:conv}) and smoothness of $\partial A$ (Proposition~\ref{prop:diff}), for each outward normal $n\in\mathbb{S}^{m-1}$ there exists a unique point $x(n)\in\partial A$ such that $(x(n),n)$ lies on the unit normal bundle of $\partial A$. Consider any initial condition $(x,n)$ for the boundary map $b$, then
    \[
    \lim_{k\to\infty} d(b^k(x,n),b^k(x(n),n))=\lim_{k\to\infty} d(L^k(x),L^k(x(n)))=0.
    \]
    because $L$ is eventually contracting.

    Hence, the distance between the $b(x,n)$-orbit of $(x,n)$ to the unit normal bundle of $\partial A$ converges to $0$.

    From the proof of Proposition~\ref{prop:conv}, we recall that $A=\lim_{k\to \infty} L^k_\varepsilon(0)$. As $L^k_\varepsilon(0)$ is convex for all $k\in\mathbb{N}$, it follows from the construction of the boundary map that
    \[
    b(N_1^+(\partial L_\varepsilon^n(0))=N_1^+\partial(L^{n+1}_\varepsilon(0)).
    \]
    Hence,
    \[
    N_1^+\partial A=\lim_{k\to\infty}b^k(N_1^+(\partial L_\varepsilon(0)).
    \]
    As $\partial L_\varepsilon(0)=\{x=\varepsilon n~|~n\in\mathbb{S}^{m-1}\}$ we finally observe that 
    \begin{eqnarray*}
    b(\{x=\varepsilon n~|~n\in\mathbb{S}^{m-1}\}) &=&
    \{x=\varepsilon L(n)+\varepsilon L_\perp(n)~|~L_\perp(n)\in\mathbb{S}^{m-1}\}\\&=&
    \{x=\varepsilon L\circ L_\perp^{-1}(n)+\varepsilon n~|~n\in\mathbb{S}^{m-1}\},
    \end{eqnarray*}
from which the expression of $N_1^+\partial A=\lim_{k\to\infty}b^k(\{x=\varepsilon n~|~n\in\mathbb{S}^{m-1}\})$ follows by induction.
\end{proof}
The explicit equation for $x(n)$ in Theorem \ref{eq:att} offers a explicit series expansion for the attractor's boundary, presenting an efficient means of boundary approximation. By systematically sampling points from the unit sphere $\mathbb{S}^{m-1}$, it becomes possible to approximate boundary points by utilising their corresponding unit normals through the series expansion. Notably, the approximation procedure converges exponentially to the actual boundary. This convergence rate aligns with that of the linear map $L$.

We now consider the dynamics on the boundary mapping $N_1^+\partial A$, which is by construction invariant. We recall that on this normal bundle, due to the strict convexity of $A$ and smoothness of its boundary, there is a one-to-one continuous relationship between points $x\in\partial A$ and its normal $n(x)$ to $\partial A$. Hence, on the one hand, this restriction of the boundary mapping is topologically conjugate to the invertible mapping induced by $L_\varepsilon$ on $\partial A$, cf.~Lemma~\ref{lem:bminv}. On the other hand, it is also topologically conjugate to $L_\perp$, the action induced by $L$ on its outward unit normal vectors. 

\begin{coro}\label{prop:top_conj}
$b|_{N_1^+\partial A}$ is topologically conjugate to $L_\perp$.     

\end{coro}
\begin{proof}

    Let $\pi (x,n):=n$ be the natural projection, then
\[ 
\pi\circ b(x,n)=L_\perp(n)=L_\perp\circ\pi(x,n).
\]
This implies the topological conjugacy, since $\pi|_{N_1^+\partial A}$ is a homeomorphism.
\end{proof}

\section{Outlook}
The results of this paper highlight some important  geometric properties of attractors of random dynamical systems with bounded noise in a very specific setting. Some interesting problems remain open, for instance the precise smoothness of the boundary of attractors. For instance, we need $C^2$ smoothness to apply the results of \cite{Kourliouros2023} to establish smooth persistence of attractor boundaries (under nonlinear perturbations), but additional smoothness requires a deeper understanding of the dynamics of the boundary map restricted to the boundary.

Finally, it would be interesting to extend our results to more general noise with non-spherical reach.

\addcontentsline{toc}{section}{References}
\bibliography{reference}

\end{document}